\documentclass[a4paper,reqno,11pt]{amsart}
\usepackage{amssymb}
\usepackage{amstext}
\usepackage{amsmath}
\usepackage{amscd}
\usepackage{amsthm}
\usepackage{amsfonts}
\usepackage{enumerate}
\usepackage{graphicx}
\usepackage{color}
\usepackage{here} % 図や表を強制的に出力させる
\usepackage{bm} % 太字ベクトル
\usepackage{mathrsfs}
\usepackage{mathtools}
\usepackage{latexsym}
\usepackage{booktabs}
\usepackage{fancyhdr}
\usepackage{subcaption}
\usepackage{tikz}
\usepackage{tikz-cd}
\usetikzlibrary{arrows}
\usetikzlibrary{decorations.markings}
\usetikzlibrary{patterns,decorations.pathreplacing}

\makeatletter
  
\@addtoreset{equation}{section}
\makeatother

\newcommand{\calA}{\mathcal{A}}

\newcommand{\calO}{\mathcal{O}}

\newcommand{\calM}{\mathcal{M}}
\newcommand{\calS}{\mathcal{S}}

\newcommand{\ZZ}{\mathbb{Z}}
\newcommand{\QQ}{\mathbb{Q}}
\newcommand{\RR}{\mathbb{R}}

\newcommand{\kk}{\Bbbk}

\newcommand{\eb}{\mathbf{e}}

\newcommand{\tb}{\mathbf{t}}

\newcommand{\Hom}{\operatorname{Hom}}

\def\opn#1#2{\def#1{\operatorname{#2}}} % to make operators
\opn\Cl{Cl} \opn\conv{conv} \opn\deg{deg} \opn\rank{rank} \opn\Spec{Spec} \opn\Stab{Stab} \opn\aff{aff} \opn\div{div} \opn\GL{GL}
\opn\cone{cone} \opn\End{End} \opn\Hom{Hom} \opn\mod{mod} \opn\gldim{gldim} \opn\pdim{pdim} \opn\diag{diag} \opn\vert{vert}
\opn\Block{Block} \opn\Pyr{Pyr} \opn\max{max} \opn\min{min} \opn\ini{in} \opn\rev{rev} \opn\ker{ker} \opn\lat{lat} \opn\pull{pull} \opn\rev{rev}

\newtheorem{thm}{Theorem}[section]
\newtheorem{cor}[thm]{Corollary}
\newtheorem{lem}[thm]{Lemma}
\newtheorem{prop}[thm]{Proposition}

\newtheorem{q}{Question}

\theoremstyle{definition}
\newtheorem{defi}[thm]{Definition}
\newtheorem{ex}[thm]{Example}

\theoremstyle{remark}

\begin{document}

\title{Torsionfreeness for divisor class groups of toric rings of integral polytopes}
\author{Koji Matsushita}

\address{Department of Pure and Applied Mathematics, Graduate School of Information Science and Technology, Osaka University, Suita, Osaka 565-0871, Japan}
\email{k-matsushita@ist.osaka-u.ac.jp}

 %52B20 Lattice polytopes in convex geometry (including relations with commutative algebra and algebraic geometry)

\maketitle

%%%---abstract---%%%
\begin{abstract} 
In the present paper, we give some sufficient conditions for  $\Cl(\kk[P])$ to be torsionfree, where $\Cl(\kk[P])$ denote the divisor class group of the toric ring $\kk[P]$ of an integral polytope $P$.
We prove that $\Cl(\kk[P])$ is torsionfree if $P$ is compressed, and  $\Cl(\kk[P])$ is torsionfree if $P$ is a $(0,1)$-polytope which has at most $\dim P+2$ facets. 
Moreover, we characterize the toric rings of $(0,1)$-polytopes in the case $\Cl(\kk[P])\cong \ZZ$.
\end{abstract}

%%%%%%%%%%%%%%%%%%%%%%%%%%%%%%%%%%%%%%%%%%%%%%%%%%%%%%%%%%%%%%%%%%%%%%%%%%%%%%%%%%%%%%%%%%%%%%%%%%%%%%%%%%%
%%%%%%%%%%%%%%%%%%%%%%%%%%%%%%%%%%%%%%%%%%%%%%%%%%%%%%%%%%%%%%%%%%%%%%%%%%%%%%%%%%%%%%%%%%%%%%%%%%%%%%%%%%%

\section{Introduction}

\subsection{Backgrounds} In the present paper, we study divisor class groups of toric rings of integral polytopes. 
Our goal is to give some sufficient conditions for them to be torsionfree.

Toric rings of integral polytopes are regarded as affine semigroup rings, and are of particular interest in the area of combinatorial commutative algebra.
Divisor class groups of Krull semigroup rings (including normal affine semigroup rings) were studied by Chouinard in \cite{Cho}. 
Especially, divisor class groups of Ehrhart rings of rational polytopes, which are normal affine semigroup rings, were also investigated by Hashimoto-Hibi-Noma \cite{HHN}. 
Ranks of  divisor class groups of Ehrhart rings are characterized in terms of polytopes associated with them, and a sufficient condition for them to be torsionfree is given.

Toric rings arise from various combinatorial objects.
In particular, the following three toric rings have been well studied:

\noindent $\bullet$ Hibi rings, which are toric rings arising from order polytopes (see, e.g., \cite{H87, S86});

\noindent $\bullet$ stable set rings, which are toric rings arising from stable set polytopes (see, e.g., \cite{Ch, HS});

\noindent $\bullet$ edge rings, which are toric rings arising from edge polytopes (see, e.g., \cite{HHO, OH98, SVV, Villa}).
Those polytopes are $(0,1)$-polytopes and the divisor class groups of their toric rings are torsionfree (\cite{HHN,HM3}). 
On the other hand, in general, divisor class groups of toric rings of integral polytopes are not necessarily torsionfree.
For example, the divisor class group of the $d$-th Veronese subring of the polynomial ring in $n$ variables over a field has torsion if $d,n\ge 2$.
It is isomorphic to the toic ring of an integral polytope.

We are also interested in the relationships among those toric rings in the case where their divisor class groups have small rank. 
Let ${\bf Order}_n$, ${\bf Stab}_n$ and ${\bf Edge}_n$ denote 
the sets of isomorphic classes of Hibi rings, stable set rings of perfect graphs and edge rings of graphs satisfying the odd cycle condition whose divisor class groups have rank $n$, respectively. Then it is shown in \cite{HM3} that 
\begin{itemize}
\item ${\bf Order}_1={\bf Stab}_1={\bf Edge}_1$. Furthermore, ${\bf Order}_1(={\bf Stab}_1={\bf Edge}_1)$ is equal to the set of the Segre products of two polynomial rings and their polynomial extensions;
\item ${\bf Stab}_2 \cup {\bf Edge}_2={\bf Order}_2$ and no inclusion between ${\bf Stab}_2$ and ${\bf Edge}_2$;
\item there is no inclusion among ${\bf Order}_3$, ${\bf Stab}_3$ and ${\bf Edge}_3$.
\end{itemize}

From the above, the following questions naturally arise:

\begin{q}\label{q1}
When are divisor class groups of toric rings of normal integral polytopes torsionfree?
\end{q}
%In particular,
\begin{q}\label{q2}
Are divisor class groups of toric rings of normal $(0,1)$-polytopes always torsionfree? Moreover, is it possible to characterize the toric rings of  normal $(0,1)$-polytopes in the case where their divisor class groups have small rank?
\end{q}

\medskip
%------------------------------------------------------------------------------------------------------------
\subsection{Results}
We give partial answers to the above two questions, as Main Result 1 and Main Result 2, respectively.
Let $P\subset \RR^d$ be an integral polytope and let $\kk[P]$ and $\Cl(\kk[P])$ denote the toric ring of $P$ over a field $\kk$ and the divisor class group of $\kk[P]$, respectively.

\medskip

\noindent {\bf Main Result 1.} 
We give a sufficient condition for $\Cl(\kk[P])$ to be torsionfree in terms of $P$ (Theorem~\ref{main1}).
Moreover, by using the consequence, we prove that $\Cl(\kk[P])$ is torsionfree if $P$ is compressed (Corollary~\ref{cor1}).

\medskip

\noindent {\bf Main Result 2.} 
We completely answer Question 2 in the case $\Cl(\kk[P])\cong\ZZ$.
Suppose that $P$ is a $(0,1)$-polytope. Then the following conditions are equivalent:
\begin{itemize}
\item[(i)] $P$ has $\dim P+2$ facets;
\item[(ii)] $\kk[P]$ is isomorphic to the Segre product of two polynomial rings over $\kk$ or its polynomial extension;
\item[(iii)] $P$ is normal and $\Cl(\kk[P])\cong\ZZ$.
\end{itemize}
In particular, $P$ is normal and $\Cl(\kk[P])$ is torsionfree if $P$ has at most $\dim P +2$ facets (Theorem~\ref{main2}).

Furthermore, we give a comment about the case $\Cl(\kk[P])\cong\ZZ^2$

\medskip

In Section 2, we recall the definitions and notation of polytopes and toric rings of integral polytopes.
We also give an algorithm to compute  divisor class groups of toric rings.
In Section 3, we provide Main Result 1 and Main Result 2.

\bigskip

%%%%%%%%%%%%%%%%%%%%%%%%%%%%%%%%%%%%%%%%%%%%%%%%%%%%%%%%%%%%%%%%%%%%%%%%%%%%%%%%%%%%%%%%%%%%%%%%%%%%%%%%%%%
%%%%%%%%%%%%%%%%%%%%%%%%%%%%%%%%%%%%%%%%%%%%%%%%%%%%%%%%%%%%%%%%%%%%%%%%%%%%%%%%%%%%%%%%%%%%%%%%%%%%%%%%%%%

\section{Preliminaries}
The goal of this section is to prepare the required materials for the discussions of our main results.
Throughout this paper, let $\kk$ be a field.

\subsection{Integral polytopes and their toric rings}

In this subsection, we recall the definitions and notation of polytopes and toric rings. We refer the readers to e.g., \cite{BG2} or \cite{Villa}, for the introduction.

\smallskip

We introduce certain polytopes:
\begin{itemize}
\item An \textit{integral polytope} (resp. \textit{rational polytope}) $P \subset \RR^d$ is a polytope whose vertices sit in $\ZZ^d$ (resp. $\QQ^d$). In particular, we call $P$ the \textit{$(0,1)$-polytope} if its all vertices are $(0,1)$-vectors. 
\item A \textit{pyramid} $P \subset \RR^d$ is the convex hull of the union of a polytope $Q \subset \RR^d$ (\textit{basis} of $P$) and a point $v_0 \in \RR^d$ (\textit{apex} of $P$), where $v_0$ does not belong to $\aff(Q)$. Note that the basis of an integral pyramid $P$ is a facet of $P$. 
\item A polytope $P\subset \RR^d$ is \textit{simple} if each vertex of $P$ is contained in precisely $\dim P$ facets.
\end{itemize}

We also define the \textit{product} of two (or more) polytopes:
for this we consider polytopes $P\subset \RR^d$ and $Q\subset \RR^e$, and set
$$P\times Q=\{(x,y) : x\in P,y\in Q\}\subset \RR^{d+e}.$$
We can see that $P\times Q$ is a polytope of dimension $\dim(P) + \dim(Q)$, whose nonempty faces
are the products of nonempty faces (including itself) of $P$ and $Q$. In particular, the number of facets of $P\times Q$ is equal to $|\Psi(P)|+|\Psi(Q)|$, where $\Psi(P)$ denotes the set of facets of $P$.

\smallskip

Simple $(0,1)$-polytopes have trivial structure. It plays an important role in our main results.

\begin{lem}[{\cite[Theorem 1]{KW}}]\label{Simple}
A $(0,1)$-polytope $P \subset \RR^d$ is simple if and only if it is equal to a product of $(0,1)$-simplices.
\end{lem}

\smallskip
%-----------------------------------------------------------------------------------------------------------
%\noindent {\bf Compressed polytopes.}
We define compressed polytopes introduced by Stanley in \cite{S80}.

\begin{defi}
An integral polytope $P$ is \textit{compressed} if every pulling triangulation of $P$ using the integral points in $P$ is unimodular.
\end{defi}

A characterization of the compressed integral polytopes is known in terms of their facet defining inequalities.
Let $\langle \cdot , \cdot \rangle$ denote the natural inner product of $\RR^d$. For $a \in \RR^d$ and $b \in \RR$, 
we denote by $H^{+}(a ; b)$ (resp. $H(a ; b)$) a closed half-space $\{ u \in \RR^d : \langle u, a\rangle + b\ge 0\}$ 
(resp. an affine hyperplane $\{ u \in \RR^d : \langle u , a \rangle +b= 0\}$). 
In particular, we denote the linear hyperplane $H(a;0)$ by $H_a$.

For each facet $F$ of an integral polytope $P \subset \RR^d$, there exist a vector $a_F \in \QQ^d$ and a rational number $b_F$ 
with the following conditions:
\begin{itemize}
\item[(i)] $H(a_F ; b_F)$ is a support hyperplane associated with $F$ and $P \subset H^{+}(a_F;b_F)$; 
\item[(ii)] $d_F(v) \in \ZZ$ for any $v \in P \cap \ZZ^d$; 
\item[(iii)] $\sum_{v\in P\cap \ZZ^d}d_F(v)\ZZ = \ZZ$,
\end{itemize}
where $d_F(v)=\langle v, a_F \rangle + b_F$ for $v \in P\cap\ZZ^d$. We can see that $d_F(v)$ is independent of the choice of $a_F$ and $b_F$.

\begin{lem}[{\cite[Theorem 2.4]{Su}}]\label{compressed}
Let $P$ be an integral polytope. Then the following conditions are equivalent:
\begin{itemize} 
\item[(i)] $P$ is compressed.
\item[(ii)] For any $F \in \Psi(P)$, $|\{d_F(v) : v \in P\cap \ZZ^d, d_F(v) \neq 0\}|=1$.
\end{itemize}
\end{lem}

Compressed polytopes appear in several places. We give some examples as follow:

\begin{itemize}
\item Order polytopes of partially ordered sets are compressed. It follows from their facet defining inequalities, which are studied in \cite{S86}.
\item For a simple graph $G$, the stable set polytope of G is compressed if and only if $G$ is perfect (\cite{GPT,OH01}).
\item Edge polytopes of bipartite graphs and complete multipartite graphs are compressed. It follows from their facet defining inequalities, which are studied in \cite{OH98}.
\end{itemize}

\medskip
%-------------------------------------------------------------------------------------------------------
%\subsection{Toric ideals and toric rings of integral polytopes}

Next, we introduce toric rings of integral polytopes. For an integral polytope $P \subset \RR^d$, we define $\phi_P$ as the morphism of $\kk$-algebras:
$$\phi_P : \kk[x_v : v\in P\cap \ZZ^d] \to \kk[t_0,t_1^{\pm 1},\ldots,t_d^{\pm 1}], \text{ induced by } \phi_P(x_v)=\tb^{v}t_0,$$

\noindent where $\tb^{v}=t_1^{v_1}\cdots t_d^{v_d}$ for $v=(v_1,\ldots,v_d) \in \ZZ^d$. 
The kernel of $\phi_P$, denoted by $I_P$, is called the \textit{toric ideal} of $P$. 
Moreover, the image of $\phi_P$, denoted by $\kk[P]$, is called  the \textit{toric ring} of $P$. 
We have $\kk[P]\cong \kk[x_v : v\in P\cap\ZZ^d]/I_P$.

The toric ideal $I_P$ is generated by homogeneous binomials. 
It is known that $P$ is compressed if and only if the initial ideal of $I_P$ with respect to any revese lexicographic monomial order on $\kk[x_v : v\in P\cap\ZZ^d]$ is generated by squarefree monomials.

The toric ring $\kk[P]$ is standard graded $\kk$-subalgebra of $\kk[t_0,t_1^{\pm 1},\ldots,t_d^{\pm 1}]$ by setting $\deg(\tb^v t_0)=1$ for each $v \in P \cap \ZZ^d$. 
The Krull dimension of $\kk[P]$, denoted by $\dim \kk[P]$, is equal to the dimension of $P$ plus 1. Namely, $\dim \kk[P]=\dim P+1$. 
We can see that the toric ring of an integral polytope $P\subset \RR^d$ is isomorphic to the affine semigroup ring of $\ZZ_{\ge 0}\calA(P)$ with coefficients in $\kk$, where $\calA(P)=\{(v,1) \in \ZZ^{d+1} : v\in P\cap \ZZ^d \}$ and $\ZZ_{\ge 0}\calA(P)=\{a_1x_1+\cdots +a_nx_n : x_1,\ldots,x_n\in \calA(P), a_1,\ldots,a_n \in \ZZ_{\ge 0} \}$.
We call $\ZZ_{\ge 0}\calA(P)$ the \textit{polytopal affine semigroup} of $P$. Note that polytopal affine semigroups are positive.

\smallskip
Let $R=\bigoplus_{n\ge 0}R_n$, $S=\bigoplus_{n\ge 0}S_n$ be two standard algebras over $\kk$ and define their \textit{Segre product} $R\# S$ as the graded algebra:
$$R\#S=(R_0\otimes_{\kk}S_0)\oplus(R_1\otimes_{\kk}S_1)\oplus\cdots \subset R\otimes_{\kk}S.$$
Let $P_1$ and $P_2$ be two integral polytopes. Then $\kk[P_1\times P_2]$ is isomorphic to the Segre product of $\kk[P_1]$ and $\kk[P_2]$.

\medskip

%---------------------------------------------------------------------------------------------------------------
\subsection{Divisor class groups of toric rings}
In this subsection, we give an algorithm to compute divisor class groups of toric rings of normal integral polytopes.
We use theories in \cite[Section 9.8]{Villa}.

\begin{prop}[cf. {\cite[Proposition 9.8.13]{Villa}}]\label{thm:equi_nomal}
For an integral polytope $P$, the following conditions are equivalent:
\begin{align*}
(\text{{\upshape i}})\; \kk[P] \text{ is normal. } \quad \quad
%(\text{{\upshape ii}})\; \kk[P] \text{ is a Krull ring. } \quad 
(\text{{\upshape ii}})\; \ZZ_{\ge 0}\calA(P)=\RR_{\ge 0}\calA(P) \cap \ZZ \calA(P).
\end{align*}
\end{prop}
We say that an integral polytope $P$ is \textit{normal} if (ii) in Theorem~\ref{thm:equi_nomal} holds.
It is known that if $P$ possesses a unimodular triangulation, then $P$ is normal (cf. {\cite[Corollary 4.12]{HHO}}). In particular, $P$ is normal if $P$ is compressed.

\smallskip

For an integral polytope $P\subset \RR^d$, there exists an irreducible representation: 
$$\RR_{\ge 0}\calA(P)=\aff(\calA(P))\cap\left(\bigcap_{F\in \Psi(P)}H^+_{(a_F,b_F)}\right).$$
Since $\langle (v,1),(a_F,b_F)\rangle=\langle v,a_F\rangle+b_F=d_F(v)$ for any $v\in P\cap\ZZ^d$ and $F\in \Psi(P)$, the following conditions are satisfied:
\begin{itemize}
\item[(a)] $\langle (v,1),(a_F,b_F)\rangle \in \ZZ \text{ for any } v\in P\cap\ZZ^d \text{ and } F\in \Psi(P);$ \vspace{0.2cm}
\item[(b)] $\ZZ=\sum_{v\in P\cap\ZZ^d}\langle (v,1),(a_F,b_F)\rangle \ZZ \text{ for all }F\in \Psi(P).$ 
\end{itemize}
Given $v \in P \cap \ZZ^d$, we define ${\bf w}_v$ belonging to a free abelian group $\bigoplus_{F \in \Psi(P)} \ZZ \eb_F$ 
with its basis $\{\eb_F\}_{F \in \Psi(P)}$ as follows: 
$${\bf w}_v=\sum_{F \in \Psi(P)}\langle (v,1),(a_F,b_F)\rangle \eb_F=\sum_{F \in \Psi(P)}d_F(v) \eb_F.$$
Let $\calS= \sum_{v \in P \cap \ZZ^d} \ZZ_{\ge 0} {\bf w}_v$. Since $\ZZ_{\ge 0}\calA(P)$ is positive, we have $\calS \cong \ZZ_{\ge 0}\calA(P)$, and hence $\rank \ZZ\calS=\dim P+1$.
Furthermore, let $\calM_P$ be the matrix whose column vectors consist of ${\bf w}_v$ for $v \in P \cap \ZZ^d$, that is, $\calM_P=\left(d_F(v)\right)_{F \in \Psi(P), v \in P\cap \ZZ^d}$. 
 
\begin{thm}[cf. {\cite[Theorem 9.8.19]{Villa}}]\label{thm:class} Work with the same notation as above and suppose that $P$ is nomal. Then 
$$\Cl(\kk[P]) \cong \bigoplus_{F \in \Psi(P)} \ZZ \eb_F \big/ \ZZ \calS.$$ 
In particular, we have 
$$\Cl(\kk[P]) \cong \ZZ^t \oplus \ZZ/s_1\ZZ \oplus \cdots \oplus \ZZ/s_r\ZZ,$$ 
where $r=\rank \calM_P=\dim P +1$, $t=|\Psi(P)|-r$ and $s_1,\ldots,s_r$ are positive integers appearing in the diagonal of the Smith normal form of $\calM_P$.
\end{thm}

The positive integers $s_1,\ldots,s_r$ are called the \textit{invariant factors} of $\calM_P$. It is known that $s_i=g_i(\calM_P)/g_{i-1}(\calM_P)$ where $g_i(\calM_P)$ denotes the greatest common divisor of all $i\times i$ minors of $\calM_P$ and $g_0(\calM_P)=1$ (see, e.g., \cite{New}).
%We call $g_i(\calM_P)$ the $i$-th \textit{determinantal divisor} of $\calM_P$.

\bigskip
%%%%%%%%%%%%%%%%%%%%%%%%%%%%%%%%%%%%%%%%%%%%%%%%%%%%%%%%%%%%%%%%%%%%%%%%%%%%%%%%%%%%%%%%%%%%%%%%%%%%%%%%%%%
%%%%%%%%%%%%%%%%%%%%%%%%%%%%%%%%%%%%%%%%%%%%%%%%%%%%%%%%%%%%%%%%%%%%%%%%%%%%%%%%%%%%%%%%%%%%%%%%%%%%%%%%%%%

\section{Main results}

We now give partial answers for Question~\ref{q1} and Question~\ref{q2}, as Main Result 1 and Main Result 2, respectively.

\subsection{Main Result 1}

Let $P \subset \RR^d$ be an integral polytope and let $k_P$ be a maximal nonnegative integer satisfying the following statement:
\begin{itemize}
\item[($*$)] There exist distinct integral points $v_1,\ldots,v_{k_P} \in P\cap \ZZ^d$ and distinct facets $F_1,\ldots,F_{k_P}$ of $P$ such that %$\dim \bigcap_{l=1}^i F_l=\dim P -i$ for each $i < \dim P +1$ 
$v_i \in \bigcap_{l=1}^{i-1} F_l$ for each $1< i \le k_P$ and $d_{F_i}(v_i)=1$ for each $1 \le i \le k_P$. \\
\end{itemize}

\begin{ex}
(a) Let $P_1=\conv(\{(0,0),(1,0),(0,1),(2,1),(1,2),(2,2)\})$. See Figure~\ref{ex1}.
We can see that $k_{P_1}=3$. Indeed, let $v_1=(1,1)$, $v_2=(1,0)$, $v_3=(0,0)$, $F_1=\conv(\{(0,0),(1,0)\})$, $F_2=\conv(\{(0,0),(0,1)\})$ and $F_3=\conv(\{(0,1),(1,2)\})$. Then we can check easily that integral points $v_1,v_2,v_3 \in P_1$ and facets $F_1,F_2,F_3 \in \Psi(P_1)$ satisfy the statement ($*$).

(b) Let $P_2=\conv(\{(1,0),(0,1),(2,1),(1,2)\})$. See Figure~\ref{ex2}. Then we have $d_F(v_1)=1$ for every facet $F$ of $P_2$ and the integral point $v_1=(1,1) \in P_2$. However, for any integral points $v$ in $P_2$ except $v_1$, the equation $d_F(v)=1$ does not hold. Thus we can obtain that $k_{P_2}=1$.     
\end{ex}

\begin{figure}[h]
\begin{minipage}{0.49\columnwidth}
\centering
\begin{tikzpicture}[domain=0:4,samples=200,>=stealth,line width=1.6pt]
\filldraw [fill=gray,draw=gray,opacity=0.4] (0,0) -- (0,1.2) -- (1.2,2.4) -- (2.4,2.4) -- (2.4,1.2) -- (1.2,0);
\draw[->] (-0.5,0) -- (3.0,0);  \draw[->] (0,-0.5) -- (0,3.0); 
\draw (0,0)--(0,1.2); \draw (0,1.2)--(1.2,2.4);\draw (1.2,2.4)--(2.4,2.4);\draw (2.4,2.4)--(2.4,1.2);\draw (2.4,1.2)--(1.2,0);\draw (1.2,0)--(0,0);
\draw (0,0) node[below left] {O};
\draw[step=1.2,black,very thin] (-0.5,-0.5) grid (2.8,2.8);
\fill[black] (0,0) circle (2.5pt) (0,1.2) circle (2.5pt)  (1.2,2.4) circle (2.5pt)  (2.4,2.4) circle (2.5pt)  (2.4,1.2) circle (2.5pt) (1.2,0) circle (2.5pt) (1.2,1.2) circle (2.5pt);
\end{tikzpicture}
\caption{The polytope $P_1$}
\label{ex1}
\end{minipage}
\begin{minipage}{0.49\columnwidth}
\centering
\begin{tikzpicture}[domain=0:4,samples=200,>=stealth,line width=1.6pt]
\filldraw [fill=gray,draw=gray,opacity=0.4] (1.2,0) -- (0,1.2) -- (1.2,2.4) -- (2.4,1.2);
\draw[->] (-0.5,0) -- (3.0,0);  \draw[->] (0,-0.5) -- (0,3.0); 
\draw (1.2,0) -- (0,1.2); \draw (0,1.2) -- (1.2,2.4); \draw (1.2,2.4) -- (2.4,1.2); \draw (2.4,1.2) -- (1.2,0);
\draw (0,0) node[below left] {O}; 
\draw[step=1.2,black,very thin] (-0.5,-0.5) grid (2.8,2.8);
\fill[black] (0,1.2) circle (2.5pt)  (1.2,2.4) circle (2.5pt) (2.4,1.2) circle (2.5pt) (1.2,0) circle (2.5pt) (1.2,1.2) circle (2.5pt);
\end{tikzpicture}
\caption{The polytope $P_2$}
\label{ex2}
\end{minipage}
\end{figure}

\begin{thm}\label{main1}
Let $P \subset \RR^d$ be a normal integral polytope and let $s_1,\ldots, s_r$ be the invariant factors of $\calM_P$.
Then $s_1=\cdots=s_{k_P}=1$. In particular, $\Cl(\kk[P])$ is torsionfree if $k_P=\dim P+1$.
\end{thm}

\begin{proof}
Assume that $v_1,\ldots,v_{k_P} \in P\cap \ZZ^d$ and $F_1,\ldots,F_{k_P} \in \Psi(P)$ satisfy the statement ($*$). Then $k_P\times k_P$ submatrix $(d_{F_i}(v_j))$ of $\calM_P$ is a triangular matrix whose diagonal entries are equal to 1 since $d_{F_1}(v_i)=\cdots=d_{F_{i-1}}(v_i)=0$ and $d_{F_i}(v_i)=1$ for each $1\le i \le k_P$. Thus we have $\det \bigl( (d_{F_i}(v_j)) \bigr)=1$, and so $g_{k_P}(\calM_P)=1$. This implies $s_1=\cdots=s_{k_P}=1$. 
Moreover, it follows directly from Theorem~\ref{thm:class} that $\Cl(\kk[P])$ is torsionfree if $k_P=\dim P +1$.
\end{proof}

This theorem give a sufficient condition for the divisor class group of toric rings to be torsionfree. 
However, it is not necessary, namely, there exists a normal integral polytope $P$ such that $\Cl(\kk[P])$ is torsionfree, but $k_P< \dim P +1$.

\begin{ex}
Let $P_3=\conv(\{(0,0),(1,4),(2,5),(3,1)\})$. See Figure~\ref{counter}. This integral polytope is normal, and we can see that $k_{P_3}=1$. However, we can compute $s_1=s_2=s_3=1$.
\end{ex}

\begin{figure}[h]
\begin{center}
\begin{tikzpicture}[domain=0:4,samples=200,>=stealth,line width=1.6pt]
\filldraw [fill=gray,draw=gray,opacity=0.4] (0,0) -- (0.75,3.0) -- (1.5,3.75) -- (2.25,0.75);
\draw[->] (-0.7,0) -- (3.45,0);  \draw[->] (0,-0.7) -- (0,4.2); 
\draw (0,0) -- (0.75,3.0); \draw (0.75,3.0) -- (1.5,3.75); \draw (1.5,3.75) -- (2.25,0.75); \draw (0,0) -- (2.25,0.75);
\draw (0,0) node[below left] {O};
\draw[step=0.75,black,very thin] (-0.7,-0.7) grid (3.35,4.1);
\fill[black] (0,0) circle (2.5pt) (0.75,3.0) circle (2.5pt)  (1.5,3.75) circle (2.5pt)  (2.25,0.75) circle (2.5pt)  (1.5,1.5) circle (2.5pt) 
(0.75,0.75) circle (2.5pt) (0.75,1.5) circle (2.5pt) (0.75,2.25) circle (2.5pt) (1.5,0.75) circle (2.5pt) (1.5,2.25) circle (2.5pt) (1.5,3.0) circle (2.5pt);
\end{tikzpicture}
\caption{The polytope $P_3$}
\label{counter}
\end{center}
\end{figure}
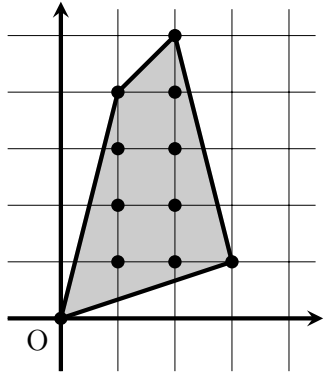

\begin{cor}\label{cor1}
If an integral polytope $P \subset \RR^d$ is compressed, then $\Cl(\kk[P])$ is torsionfree.
\end{cor}

\begin{proof}
Let $n=\dim P$. We can choose $F_1,\ldots,F_{n+1} \in \Psi(P)$ with $\dim \bigcap_{l=1}^i F_l=n -i$ for each $1\le i \le n$ and $\bigcap_{l=1}^n F_l \nsubseteq F_{n+1}$, and choose $v_1,\ldots,v_{n+1} \in P\cap \ZZ^d$ with $v_i \notin F_i$ for each $1\le i \le n+1$ and $v_{i+1} \in \bigcap_{l=1}^i F_l$ for each $1\le i \le n$. By Lemma~\ref{compressed} and $\sum_{v \in P \cap \ZZ^d}d_F(v)\ZZ=\ZZ$ for any $F\in \Psi(P)$, one has $d_{F_i}(v_i)=1$ for each $1\le i \le n$. 
Therefore the $v_i$'s and $F_i$'s satisfy the statement ($*$).
\end{proof}

\medskip

%------------------------------------------------------------------------------------------------------------
\subsection{Main Result 2}

%We prepare propositions associated with toric rings of $(0,1)$-polytopes.

\begin{lem}[cf. {\cite[Lemma 9.3.2]{Villa}}]\label{d+1}
Let $\Delta \subset \RR^d$ be a $ (0,1)$-simplex of dimension $n$. Then $\kk[\Delta]$ is isomorphic to the polynomial ring with $n+1$ variables over $\kk$.
\end{lem}

\begin{lem}\label{pyramid}
Let $P \subset \RR^d$ be an integral pyramid with basis $Q$ and apex $v_0$.
If $(P\setminus \{v_0\}) \cap \ZZ^d \subset Q$, then 
$$\kk[P] \cong \kk[Q] \otimes \kk[x_{v_0}].$$
In particular, $\Cl(\kk[P])\cong \Cl(\kk[Q])$ if $P$ is normal.
\end{lem}

\begin{proof}
%Since $(P\setminus \{v_0\}) \cap \ZZ^d \subset Q$, it is enough to show that $I_P\cap (x_v : v \in Q\cap \ZZ^d)=I_Q$. 
Let $f_1,\ldots,f_s$ be a minimal system of generators of $I_P$ and let $H(a_Q; b_Q)$ be a hyperplane which defines the facet $Q$ of $P$ with $P \subset H^{(+)}(a_Q;b_Q)$. Then $f_j$ is a homogeneous binomial, thus we can write $f_j=x_{u_{j1}}\cdots x_{u_{jn}}-x_{v_{j1}}\cdots x_{v_{jn}}$ where $u_{j1},\ldots, u_{jn}, v_{j1}, \ldots, v_{jn} \in P\cap \ZZ^d$ with $\sum_{i=1}^n u_{ji}=\sum_{i=1}^n v_{ji}$ and $\{u_{j1},\ldots,u_{jn}\}\cap \{v_{j1},\ldots,v_{jn}\}=\emptyset$. We may assume that $v_0 \notin \{v_{j1},\ldots,v_{jn}\}$. 

We calculate $\sum_{i=1}^n \left( \langle u_{ji},a_Q \rangle +b_Q \right) - \sum_{i=1}^n \left( \langle v_{ji},a_Q \rangle +b_Q \right)$:
\begin{align*}
\sum_{i=1}^n \left( \langle u_{ji},a_Q \rangle +b_Q \right) - \sum_{i=1}^n \left( \langle v_{ji},a_Q \rangle +b_Q \right)= \langle \sum_{i=1}^n u_{ji}-\sum_{i=1}^n v_{ji} , a_Q \rangle=0.
\end{align*}
On the other hand, if $v_0 \in \{u_{ji},\ldots,u_{ji}\}$, then 
\begin{align*}
\sum_{i=1}^n \left( \langle u_{ji},a_Q \rangle +b_Q \right) - \sum_{i=1}^n \left( \langle v_{ji},a_Q \rangle +b_Q \right)= \sum_{i=1}^n \left( \langle u_{ji},a_Q \rangle +b_Q \right) > 0,
\end{align*}
since $v_{ji}\in Q$ and $\langle v_{ji},a_Q \rangle +b_Q=0$ for each $1\le i \le n$, a contradiction. Therefore $v_0 \notin \{u_{ji},\ldots,u_{ji}\}$, that is, $f_1,\ldots,f_s \in I_Q$ and 
$$\kk[P]\cong \kk[x_v: v\in P\cap \ZZ^d]/I_P\cong \kk[x_v : v\in Q\cap \ZZ^d]/I_Q\otimes \kk[x_{v_0}]\cong \kk[Q]\otimes \kk[x_{v_0}].$$
\end{proof}

\begin{thm}\label{main2}
Let $P \subset \RR^d$ be a $(0,1)$-polytope. Then the following conditions are equivalent:
\begin{itemize}
\item[(i)] $P$ has $\dim P+2$ facets;
\item[(ii)] $\kk[P]$ is isomorphic to the Segre product of two polynomial rings over $\kk$ or its polynomial extension; %$\kk[x_1,\ldots,x_s]$ and $\kk[y_1,\ldots,y_t]$ for some $s,t \in \ZZ_{>0}$ 
\item[(iii)] $P$ is normal and $\Cl(\kk[P]) \cong \ZZ$.
\end{itemize}
In particular, $P$ is normal and $\Cl(\kk[P])$ is torsionfree if $P$ has at most $\dim P+2$ facets.
\end{thm}

\begin{proof}
(i)$\Rightarrow$(ii): A polytope $P$ which has $\dim P+2$ facets is a simple polytope or a pyramid. 
If $P$ is an integral pyramid with basis $Q$, then $Q$ has $\dim Q +2$ facets. Thus $Q$ is a simple polytope or a pyramid again. Therefore, we consier $P$ in the case where $P$ is simple. From Lemma~\ref{Simple}, $P$ is equal to a product of $(0,1)$-simplices $\Delta_1,\ldots,\Delta_m$. Since $\Delta_i$ has $\dim \Delta_i +1$ facets for each $i$, we see that $P$ has $\sum_{i=1}^m (\dim \Delta_i +1)=\dim P +m$ facets. Thus, we have $m=2$, and hence $\kk[P]$ is isomorphic to the Segre product of two polynomial rings or its polynomial extension by Lemma~\ref{d+1} and Lemma~\ref{pyramid}.

(ii)$\Rightarrow$(iii): It is known that the Segre product of two polynomial rings over $\kk$ is normal and its divisor class group is isomorphic to $\ZZ$. 
In fact, the Segre product of some polynomial rings is realized as a Hibi ring (see, e.g., \cite[Example 2.6]{HN}). 
Hibi rings are normal (\cite{H87}) and the description of their divisor class groups is provided in [5].
Thus, $P$ is normal from Proposition~\ref{thm:equi_nomal} and $\Cl(\kk[P])\cong\ZZ$.

(iii)$\Rightarrow$(i): From Theorem~\ref{thm:class}, the rank of $\Cl(\kk[P])$ is equal to $|\Psi(P)|-(\dim P+1)$. Therefore $|\Psi(P)|=\dim P +2$.
\end{proof}

\smallskip

Recall that ${\bf Order}_n$, ${\bf Stab}_n$ and ${\bf Edge}_n$ are the sets of isomorphic classes of Hibi rings, stable set rings of perfect graphs and edge rings of graphs satisfying the odd cycle condition whose divisor class groups have rank $n$, respectively.
From Theorem~\ref{main2}, the set of isomorphic classes of toric rings of normal $(0,1)$-polytopes whose divisor class groups have rank 1 is equal to ${\bf Order}_1(={\bf Stab}_1={\bf Edge}_1)$.

We completely characterize the toric rings of normal $(0,1)$-polytopes $P$ such that $\Cl(\kk[P])\cong\ZZ$.
But, the characterization of the case $\Cl(\kk[P])\cong\ZZ^2$ is more difficult.
In fact, while ${\bf Stab}_2\cup {\bf Edge}_2 \subset {\bf Order}_2$, there exist a normal $(0,1)$-polytope $P$ such that $\Cl(\kk[P])\cong\ZZ^2$ and $\kk[P]\notin {\bf Order}_2$.

\begin{ex}
Let $P=\conv(\{(0,0,0),(1,0,0),(0,1,0),(0,0,1),(1,1,1)\})$. 
Then the set of support hyperplanes of $P$ is:
$$\{H(\eb_1;0), H(\eb_2;0),H(\eb_3;0),H(\eb_1-\eb_2-\eb_3;1),H(-\eb_1+\eb_2-\eb_3;1),H(-\eb_1-\eb_2+\eb_3;1)\},$$
where $\eb_i$ denotes the $i$-th unit vector of $\RR^3$ for $i\in \{1,2,3\}$.
We can see that $P$ is normal and $\Cl(\kk[P])\cong\ZZ^2$. On the other hand, if an order polytope $\calO$ satisfies $\dim \calO=3$ and $\kk[\calO]\in {\bf Order}_2$, then $\calO$ is the 3-dimensional unit hypercube.
Therefore, we have $\kk[P]\notin {\bf Order}_2$.
\end{ex}

%%%%%%%%%%%%%%%%%%%%%%%%%%%%%%%%%%%%%%%%%%%%%%%%%%%%%%%%%%%%%%%%%%%%%%%%%%%%%%%%%%%%%%%%%%%%%%%%%%%%%%%%%%%%%
%%%%%%%%%%%%%%%%%%%%%%%%%%%%%%%%%%%%%%%%%%%%%%%%%%%%%%%%%%%%%%%%%%%%%%%%%%%%%%%%%%%%%%%%%%%%%%%%%%%%%%%%%%%%%

%%%%%%%%%%%%%%%%%%%%%%%%%%%%%%%%%%%%%%%%%%%%%%%%%%%%%%%%%%%%%%%%%%%%%%%%%%%%%%%%%%%%%%%%%%%%%%%%%%%%%%%%%%%%
%%%%%%%%%%%%%%%%%%%%%%%%%%%%%%%%%%%%%%%%%%%%%%%%%%%%%%%%%%%%%%%%%%%%%%%%%%%%%%%%%%%%%%%%%%%%%%%%%%%%%%%%%%%%
\end{document}